\documentclass[12pt,final]{amsart}
\usepackage{amsmath,amsthm,amssymb}
\usepackage{amsfonts}
\usepackage[varg]{txfonts}
\usepackage[mathscr]{eucal}
\usepackage{bxpapersize}
\usepackage{ascmac}
\usepackage{epic, eepic}
\usepackage{graphicx}
\usepackage{indentfirst}
\usepackage{cases}
\usepackage{color}
\usepackage{url}
\usepackage{bm}
\usepackage[all]{xy}
\usepackage[top=25truemm,bottom=25truemm,left=30truemm,right=30truemm]{geometry}
\usepackage[noadjust]{cite}
\usepackage{showkeys}
\usepackage{enumerate}
\usepackage{hyperref}
\hypersetup{
setpagesize=false,
 bookmarksnumbered=true,%
 bookmarksopen=true,%
 colorlinks=true,%
 linkcolor=blue,
 citecolor=blue,
 urlcolor=magenta,
}

\theoremstyle{plane}

\newtheorem{thm}{Theorem}[section]
\newtheorem{prop}[thm]{Proposition}
\newtheorem{lem}[thm]{Lemma}
\newtheorem{cor}[thm]{Corollary}
\newtheorem{fact}[thm]{Fact}
\theoremstyle{definition}
\newtheorem{dfn}[thm]{Definition}
\newtheorem{ex}[thm]{Example}

\theoremstyle{remark}
\newtheorem{rem}[thm]{Remark}
 \newtheorem*{acknowledgements}{Acknowledgements}

\newcommand{\rank}{\operatorname{rank}}

\newcommand{\re}{\operatorname{Re}}
\newcommand{\im}{\operatorname{Im}}

\newcommand{\sgn}{\operatorname{sgn}}

\newcommand{\R}{\bm{R}}
\newcommand{\C}{\bm{C}}

\newcommand{\n}{\bm{n}}
\newcommand{\x}{\bm{x}}

\newcommand{\Sig}{\Sigma}
\renewcommand{\phi}{\varphi}
\newcommand{\eps}{\varepsilon}

\newcommand{\inner}[2]{\left\langle{#1},{#2}\right\rangle}
\newcommand{\linner}[2]{\left\langle{#1},{#2}\right\rangle_L}

\numberwithin{equation}{section}

\title[Geometric properties of surfaces given by certain formulae]
{Geometric properties near singular points of surfaces given by certain representation formulae}
\author[Y. Matsushita]{Yoshiki Matsushita}
\address[Y. Matsushita]{Department of Mathematics, Kyushu University, 744 Motooka, Fukuoka, 819-0395, Japan}
\email{ma218015@math.kyushu-u.ac.jp}
\author[T. Nakashima]{Takuya Nakashima}
\address[T. Nakashima]{Department of Mathematics, Kyushu University, 744 Motooka, Fukuoka, 819-0395, Japan}
\email{2ma18040m@kyudai.jp}
\author[K. Teramoto]{Keisuke Teramoto}
\address[K. Teramoto]{Institute of Mathematics for Industry, Kyushu University, 744 Motooka, Fukuoka 819-0395, Japan}
\email{k-teramoto@imi.kyushu-u.ac.jp}

\thanks{The third author was partially supported by Grant-in-Aid for JSPS KAKENHI Grant Number JP19K14533.}
\date{\today}

\begin{document}

\maketitle

\begin{abstract}
We investigate geometric properties of surfaces given by certain formulae. 
In particular, we calculate the singular curvature and the limiting normal curvature of such surfaces 
along the set of singular points consisting of singular points of the first kind. 
Moreover, we study fold singular points of smooth maps. 
\end{abstract}

\section{Introduction}\label{sec:intro}
Let $\R^3$ be the Euclidean $3$-space with canonical inner product $\inner{\cdot}{\cdot}$, 
and let $D$ be a simply-connected domain in the complex plane $(\C;z=u+iv)$, where $i=\sqrt{-1}$.
 Then we consider two representation formulae for maps from $D$ to $\R^3$. 
To state the first representation formula, we give a definition. 
\begin{dfn}\label{def:hol-data}
Let $g$ be a meromorphic function and $\omega=\hat{\omega}dz$ a holomorphic $1$-form defined on $D$. 
Then the pair $(g,\omega)$ is said to be the {\it holomorphic data} if 
it satisfies the following conditions: 
\begin{enumerate}
\item $g^2\omega$ is holomorphic on $D$. 
\item $(1+|g|^2)^2|\omega|^2\neq0$ on $D$. 
\item $1-|g|^2$ does not vanish identically. 
\end{enumerate}
\end{dfn}

Using the holomorphic data $(g,\omega)$, 
we define a map $f\colon D\to\R^3$ by 
\begin{equation}\label{eq:W-rep}
f = \re\left(\int(-2g,1+g^2, i(1-g^2))\omega\right).
\end{equation}
It is known that if we consider $f$ as a map to the Minkowski $3$-space $\R^3_1$
with pseudo-inner product $\linner{\bm{x}}{\bm{y}}=-x_1y_1+x_2y_2+x_3y_3$ 
for any $\bm{x}=(x_1,x_2,x_3)$ and $\bm{y}=(y_1,y_2,y_3)\in\R^3$, 
then this is a Weierstrass-type representation formula for a {\it maxface}, which is a maximal surface 
(i.e., spacelike zero mean curvature surface) with certain singularities introduced in \cite{maxface}. 
In such a case, the pair $(g,\omega)$ is called the {\it Weierstrass data} of a maxface. 
There are several studies on maximal surfaces and maxfaces (\cite{gene-max,ferlop,fkkrsuyy,fkkruy2,fkkruy,fsuy,kimyang1,kimyang2,koba,ku,maxface}). 
It is known that generic singularities of maxfaces are a cuspidal edge, a swallowtail and a cuspidal cross cap \cite{fsuy}. 
Moreover, there are maxfaces with a cuspidal butterfly and a cuspidal $S_1^-$ singularity \cite{co2,ot}. 
Criteria for these singularities using the data $(g,\omega)$ are known \cite{fsuy,ot,maxface}. 
Moreover, we show that there are no surfaces given by \eqref{eq:W-rep} with cuspidal $S_k$ singularities for $k\geq2$ (Theorem \ref{prop:no-cSk}). 

To state the other one, we define some notions. 
\begin{dfn}\label{def:generalized-cmc}
Let $\hat{\C}=\C\cup\{\infty\}$ be the Riemann sphere. 
Let $D\subset\C$ be a simply-connected domain and let $g\colon D\to\hat{\C}$ be a smooth map.  
Then $g$ is said to be a {\it regular extended harmonic map} if 
\begin{enumerate}
\item $g_{z\overline{z}}+2(1-|g|^2)\overline{g}g_{z}\overline{\hat{\omega}}=0$ holds, 
\item $\omega=\hat{\omega}dz$ can be extended to a $1$-form of class $C^1$ across $\{p\in D\ |\ |g(p)|=1\}$,
\end{enumerate}
where 
\begin{equation}\label{eq:cmc-omega}
\hat{\omega}=\dfrac{\overline{g}_z}{(1-|g|^2)^2}
\end{equation}
and $z$ is a complex coordinate of $D$. 
Moreover, a regular extended harmonic function $g$ is called an {\it extended harmonic map} if the following conditions hold: 
\begin{itemize}
\item $\hat{\omega}$ never vanishes on $\{p\in D\ |\ |g(p)|<\infty\}$, and 
\item $g^2\hat{\omega}$ does not vanish on $\{p\in D\ |\ |g(p)|=\infty\}$.
\end{itemize}
\end{dfn}
These notions are introduced in \cite{umeda}. 
Using an extended harmonic map $g\colon D\to\hat{\C}$ and a non-zero constant $H$, 
we define a map $f\colon D\to\R^3$ by 
\begin{equation}\label{eq:generalized-cmc}
f=\dfrac{2}{H}\re\left(\int(-2g,1+g^2,i(1-g^2))\hat{\omega}dz\right).
\end{equation}
If we treat $f$ as in \eqref{eq:generalized-cmc} as a map from $D$ to $\R^3_1$, 
this gives a Kenmotsu-type representation formula for an {\it extended spacelike constant mean curvature $($\/CMC\/$)$ $H$ surface}, 
which is a spacelike CMC surface with certain singularities (cf. \cite{an,kenmotsu,umeda}). 
It is known that criteria for a cuspidal edge, a swallowtail and a cuspidal cross cap on extended spacelike CMC surfaces 
are given in terms of $g$ and $\hat{\omega}$ (\cite{umeda}).  

As we mentioned above, surfaces given by \eqref{eq:W-rep} or \eqref{eq:generalized-cmc} have singularities in general. 
Moreover, they belong to the class of singular surfaces in $\R^3$ called a {\it frontal} or a {\it front}, 
which admit the unit normal vector field even at singular points (see Section \ref{sec:frontal}). 
Singular points such as cuspidal edges, cuspidal cross caps and cuspidal $S_k$ singularities 
belong to the class of {\it singular points of the first kind} of frontal surfaces (cf. \cite{msuy}). 
The set of such singular points consists of  regular curves ( called {\it singular curves}) on the source 
and their images are regular space curves. 
Along such curves, several geometric invariants are introduced ({\cite{ms,msuy,geomfront}}). 
In particular, the {\it singular curvature} $\kappa_s$ and the {\it limiting normal curvature} $\kappa_\nu$ are representative 
because they satisfy $\kappa^2=\kappa_s^2+\kappa_\nu^2$, 
where $\kappa$ is the curvature of a singular image as a space curve in $\R^3$ (cf. \cite{ms,msuy}). 
Moreover, $\kappa_s$ is an {\it intrinsic invariants} of a frontal, 
and its sign relates to the {\it convexity} and {\it concavity} (see Figure \ref{fig:ks}) (\cite{hhnsuy,geomfront}). 
Further, $\kappa_\nu$ relates to the boundedness of the Gaussian curvature of a frontal near a non-degenerate singular point (\cite{msuy,geomfront}).

\begin{figure}[htbp]
  \begin{center}
    \begin{tabular}{c}

      \begin{minipage}{0.33\hsize}
        \begin{center}
          \includegraphics[width=3.5cm]{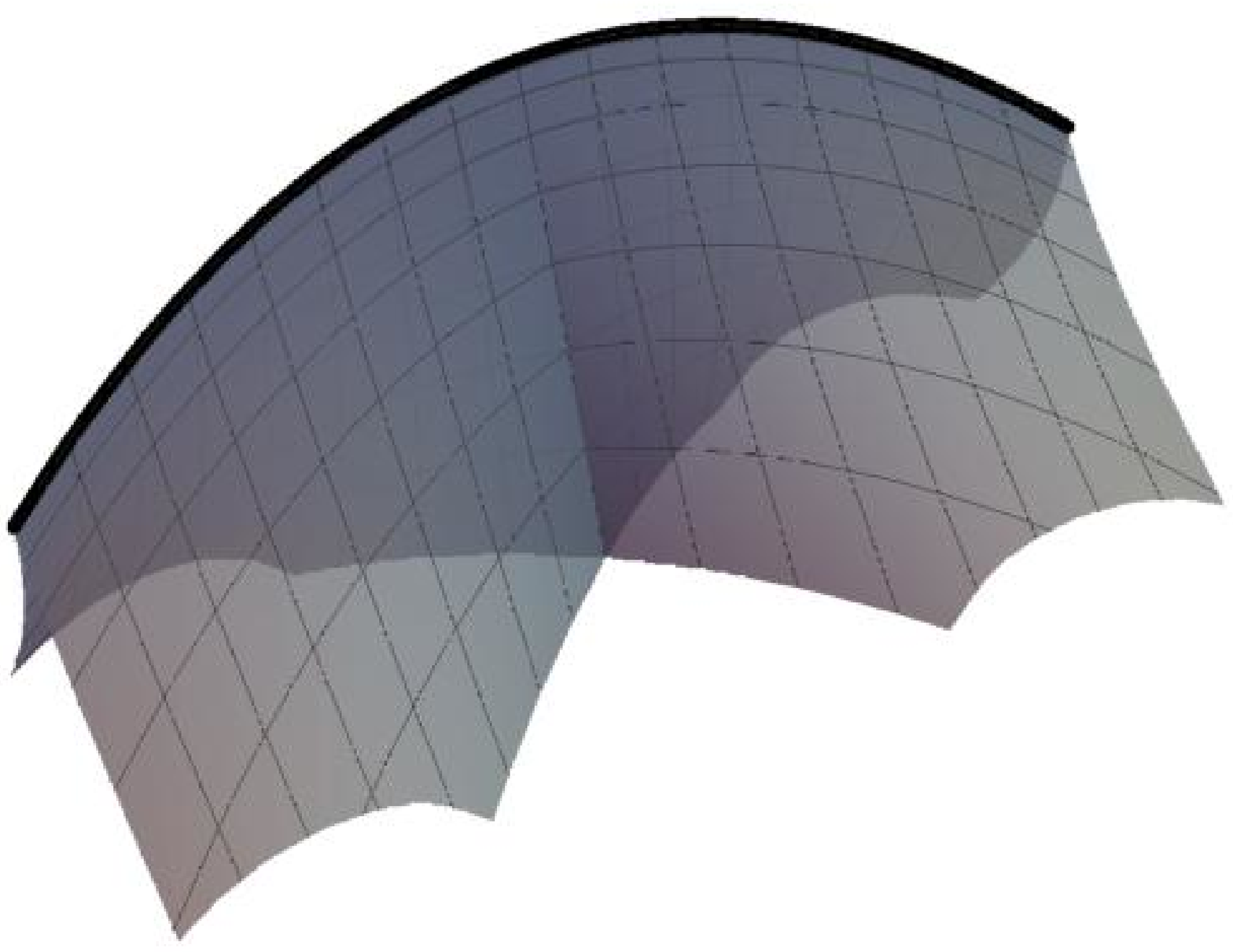}
        \end{center}
      \end{minipage}

      \begin{minipage}{0.33\hsize}
        \begin{center}
          \includegraphics[width=3cm]{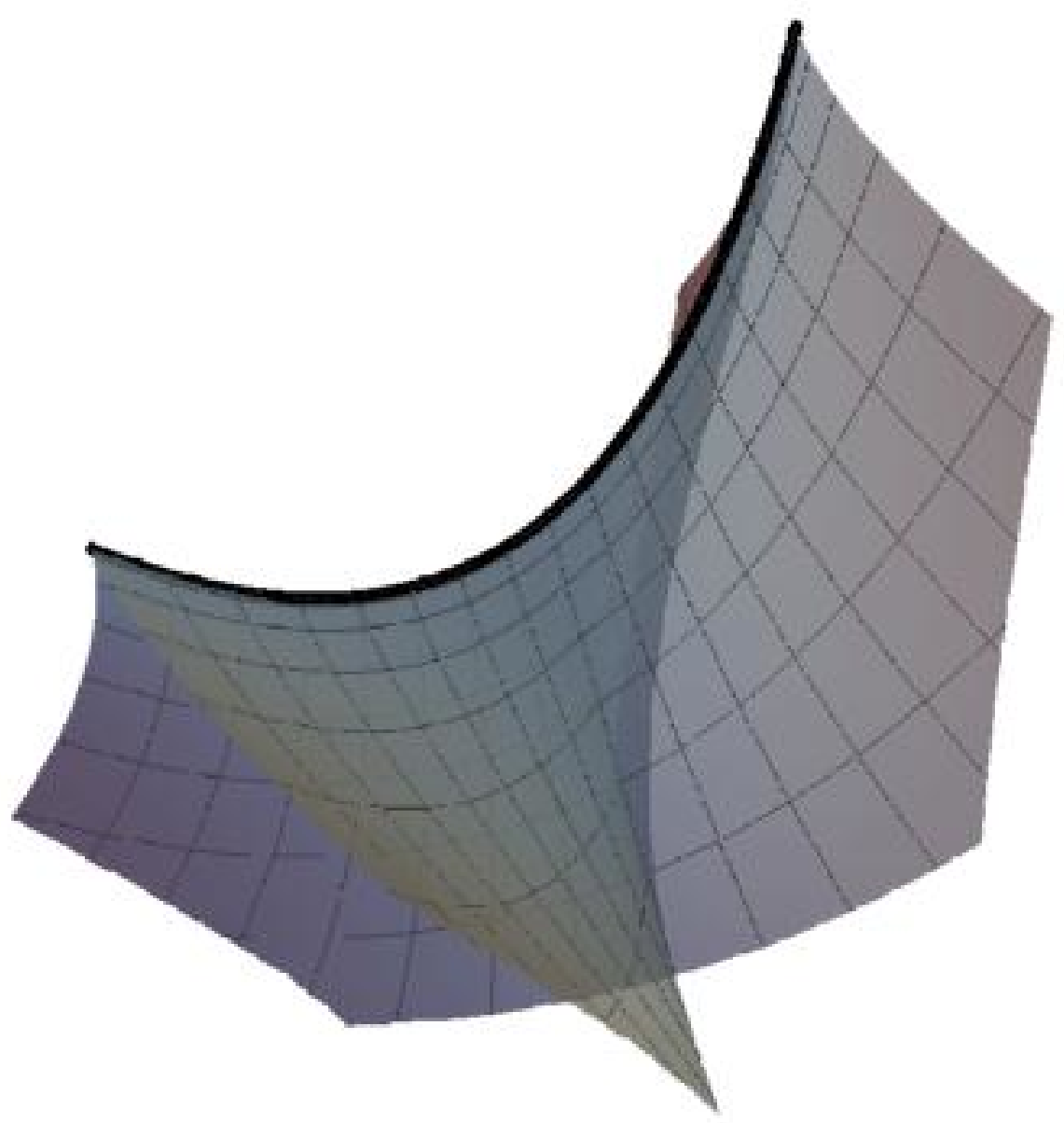}
        \end{center}
      \end{minipage}

    \end{tabular}
    \caption{Cuspidal cross caps with positive $\kappa_s$ (left) and negative $\kappa_s$ (right).}
    \label{fig:ks}
  \end{center}
\end{figure}

In this paper, we study geometric properties of surfaces given by \eqref{eq:W-rep} or \eqref{eq:generalized-cmc} in $\R^3$ 
near a singular point of the first kind. 
In particular, we focus on the singular curvature and the limiting normal curvature at singular points of such surfaces.
We show that the singular curvature is negative along the singular curve consisting of singular points of the first kind, 
and the limiting normal curvature vanishes along the singular curve for surfaces given 
by \eqref{eq:W-rep} or \eqref{eq:generalized-cmc} (Theorem \ref{thm:ks}). 
This implies that the Gaussian curvature $K_{\mathrm{E}}$ 
of a frontal in $\R^3$ given by \eqref{eq:W-rep} or \eqref{eq:generalized-cmc} is bounded 
near non-degenerate singular points by the general theory (\cite{msuy,geomfront}). 
Especially, we show relationships between signs of the Gaussian curvature and the singular curvature 
near singular points of the first kind (Corollary \ref{cor:ks-Gauss}). 
Moreover, one can see that the Gauss map $\n$ of a surface $f\colon D\to\R^3$ given by \eqref{eq:W-rep} or \eqref{eq:generalized-cmc} 
has singularities along the singular curve of $f$. 
Thus we investigate types of singularities of $\n$ (Proposition \ref{prop:sing-n}).

Further, we give a certain characterization for a fold singular point of a frontal surface (Theorem \ref{thm:crit-fold}).

\section{Preliminaries}\label{sec:prelim}
\subsection{Frontal}\label{sec:frontal}
We review some notions of frontal surfaces quickly. 
For details, see \cite{a,agv,ifrt,krsuy,msuy,ms,geomfront}.

Let $f\colon\Sig\to\R^3$ be a $C^\infty$ map, where $\Sig$ is an open domain in $\R^2$ 
and $\R^3$ is the Euclidean $3$-space with canonical inner product $\inner{\cdot}{\cdot}$. 
Then $f$ is a {\it frontal surface} (or a {\it frontal} for short) 
if there exists a $C^\infty$ map $\n\colon\Sig\to S^2$ such that 
$\inner{df_q(X)}{\n(q)}=0$ holds for any $q\in \Sig$ and $X\in T_q\Sig$, where $S^2$ denotes the standard unit sphere in $\R^3$. 
The map $\n$ is called a {\it unit normal vector} or the {\it Gauss map} of $f$. 
A frontal $f$ is called a {\it front} if the pair $(f,\n)\colon\Sig\to\R^3\times S^2$ gives an immersion. 

We fix a frontal $f$. 
A point $p\in \Sig$ is a {\it singular point} of $f$ if $\rank df_p<2$ holds. 
We denote by $S(f)$ the set of singular points of $f$. 
On the other hand, we define a function $\lambda\colon\Sig\to\R$ by 
\begin{equation}\label{eq:lambda}
\lambda(u,v)=\det(f_u,f_v,\n)(u,v),
\end{equation}
where $(u,v)$ is a local coordinate system on $\Sig$. 
This function $\lambda$ is called the {\it signed area density function}. 
For the function $\lambda$, it is known that there exist functions $\hat{\lambda}$ and $\mu$ such that 
$\lambda=\hat{\lambda}\cdot\mu$, $\hat{\lambda}^{-1}(0)=S(f)$ and $\mu>0$ on $\Sig$. 
We call $\hat{\lambda}$ the {\it singularity identifier} of $f$.  

A singular point $p\in S(f)$ of a frontal $f$ is {\it non-degenerate} if $(\hat{\lambda}_u(p),\hat{\lambda}_v(p))\neq(0,0)$. 
Take a non-degenerate singular point $p$. 
Then there exist a neighborhood $U(\subset \Sig)$ of $p$ and a regular curve $\gamma=\gamma(t)\colon(-\eps,\eps)\to U$ $(\eps>0)$
with $\gamma(0)=p$ such that $\hat{\lambda}(\gamma(t))=0$ on $U$ by the implicit function theorem. 
We call the curve $\gamma$ a {\it singular curve}. 
Further, we call the image $\hat{\gamma}=f\circ\gamma$ of a singular curve $\gamma$ by $f$ a {\it singular locus}. 
Moreover, since a non-degenerate singular point $p$ satisfies $\rank df_p=1$, 
there exists a non-zero vector field $\eta$ on $U$ such that $df_q(\eta_q)=0$ for any $q\in S(f)\cap U$. 
This vector field $\eta$ is called a {\it null vector field}. 
Further, one can take a vector field $\xi$ on $U$ so that $\xi$ is tangent to $\gamma$ on $S(f)\cap U$. 
We call the direction of $\xi$ along $\gamma$ the {\it singular direction}. 
A non-degenerate singular point $p$ is of the {\it first kind} if $\xi$ and $\eta$ are linearly independent at $p$. 
Otherwise, it is said to be of the {\it second kind}. 

\begin{dfn}\label{def:singularities}
\begin{enumerate}
\item Let $f,g\colon(\R^m,0)\to(\R^n,0)$ be $C^\infty$ map-germs. 
Then $f$ and $g$ are {\it $\mathcal{A}$-equivalent} 
if there exist diffeomorphism-germs $\phi\colon(\R^m,0)\to(\R^m,0)$ on the source 
and $\Phi\colon(\R^n,0)\to(\R^n,0)$ on the target such that $\Phi\circ f\circ\phi^{-1}=g$ holds. 
\item Let $f\colon(\R^2,0)\to(\R^3,0)$ be a $C^\infty$ map-germ. 
Then 
\begin{itemize}
\item $f$ at $0$ is a {\it cuspidal edge} 
if $f$ is $\mathcal{A}$-equivalent to the germ 
$(u,v)\mapsto(u,v^2,v^3)$ at $0$. 
\item $f$ at $0$ is a {\it swallowtail} if $f$ is $\mathcal{A}$-equivalent to the germ $(u,v)\mapsto(u,3v^4+uv^2,4v^3+2uv)$ at $0$. 
\item $f$ at $0$ is a {\it cuspidal butterfly} if $f$ is $\mathcal{A}$-equivalent to the germ 
$(u,v)\mapsto(u,4v^5+uv^2,5v^4+2uv)$ at $0$. 
\item $f$ at $0$ is a {\it cuspidal cross cap} if $f$ is $\mathcal{A}$-equivalent to the germ 
$(u,v)\mapsto(u,v^2,uv^3)$ at $0$.
\item $f$ at $0$ is a {\it cuspidal $S_k^\pm$ singularity} $(k\geq1)$ if $f$ is $\mathcal{A}$-equivalent to the germ 
$(u,v)\mapsto(u,v^2,v^3(u^{k+1}\pm v^2))$ at $0$. 
\item $f$ at $0$ is a {\it $5/2$-cuspidal edge} if $f$ is $\mathcal{A}$-equivalent to the germ 
$(u,v)\mapsto(u,v^2,v^5)$ at $0$. 
\end{itemize}
\end{enumerate}
\end{dfn}
We note that these singularities are all non-degenerate frontal singularities. 
Moreover, a cuspidal edge, a cuspidal cross cap, a cuspidal $S_k^\pm$ singularity and a $5/2$-cuspida edge are 
of the first kind, but a swallowtail and a cuspidal butterfly are of the second kind. 
Criteria for these singularities are known (see \cite{hks,krsuy,Aksingul,saji}).
We remark that certain dualities of singularities for maxfaces and generalized spacelike CMC surfaces are known (\cite{fsuy,umeda,hks}). 

We next consider the geometric invariants of a frontal at a singular point of the first kind. 
As discussions above, one can take $\xi$ and $\eta$ around a singular point of the first kind. 
Using these vector fields, we define two geometric invariants as follows:
\begin{align}\label{eq:invariants}
\begin{aligned}
\kappa_s(\gamma(t))&=\left.\eps_{\gamma}\dfrac{\det(\xi f,\xi\xi f,\n)}{|\xi f|^3}\right|_{(u,v)=\gamma(t)},\quad 
\kappa_\nu(\gamma(t))=\left.\dfrac{\inner{\xi\xi f}{\n}}{|\xi f|^2}\right|_{(u,v)=\gamma(t)},\\
\end{aligned}
\end{align}
where $\eps_{\gamma}=\sgn(\det(\gamma',\eta)\cdot\eta\lambda)=\sgn(\det(\xi,\eta)\cdot\eta\hat{\lambda})$ 
along the singular curve $\gamma$ and $|\cdot|$ is the standard norm of $\R^3$. 
The invariants $\kappa_s$ and $\kappa_\nu$ are called the {\it singular curvature} and the {\it limiting normal curvature}, respectively. 
We remark that $\kappa_s$ is an intrinsic invariant and its sign has a geometrical meaning (\cite{hhnsuy,geomfront}). 
Moreover, $\kappa_\nu$ relates to the behavior of the Gaussian curvature (\cite{msuy}). 
For more details and other invariants at singular points of the first kind, see \cite{hhnsuy,ms,nuy,msuy,geomfront}. 

\subsection{Singularities of surfaces given by certain representations}\label{sec:sing-surface}
We recall singularities of surfaces given by \eqref{eq:W-rep} or \eqref{eq:generalized-cmc}. 
Since types of singular points do not depend on the metric, 
we consider that the target space is the Euclidean $3$-space $\R^3$. 
\subsubsection{The case of a surface given by \eqref{eq:W-rep}}\label{sec:sing-W}
Let $f\colon D\to\R^3$ be a map given by \eqref{eq:W-rep} 
with the holomorphic data $(g,\omega=\hat{\omega}dz)$ on a simply-connected domain $D\subset \C$. 
By a direct calculation, the differentials of $f$ by $z$ and $\overline{z}$ are 
\begin{equation}\label{eq:df}
f_z=\dfrac{1}{2}(-2g,1+g^2,i(1-g^2))\hat{\omega},\quad 
f_{\overline{z}}=\dfrac{1}{2}(-2\overline{g},1+\overline{g}^2,-i(1-\overline{g}^2))\overline{\hat{\omega}}.
\end{equation}
We note that $f_{z}$ is a holomorphic map with respect to $z$. 
Since 
$\partial_z=(\partial_u-i\partial_v)/2$ and $\partial_{\overline{z}}=(\partial_u+i\partial_v)/2$, 
we have 
\begin{equation}\label{eq:cross}
f_u\times f_v=-2if_z\times f_{\overline{z}}=(|g|^2-1)|\hat{\omega}|^2(1+|g|^2,2\re(g),2\im(g)),
\end{equation}
where $\times$ is the canonical vector product of $\R^3$.
Thus the unit normal vector $\n$ of $f$ can be taken as 
\begin{equation}\label{eq:normal}
\n=\dfrac{1}{\sqrt{(1+|g|^2)^2+4|g|^2}}(1+|g|^2,2\re(g),2\im(g)).
\end{equation}
Moreover, by \eqref{eq:cross} and \eqref{eq:normal}, the signed area density function $\lambda$ of $f$ is 
\begin{equation*}
\lambda=(|g|^2-1)|\hat{\omega}|^2\sqrt{(1+|g|^2)^2+4|g|^2}.
\end{equation*}
Since $|\hat{\omega}|^2\sqrt{(1+|g|^2)^2+4|g|^2}>0$, $S(f)=\{p\in D\ |\ |g(p)|=1\}$ (cf. \cite{maxface}) 
and the singularity identifier $\hat{\lambda}$ of $f$ is 
\begin{equation}\label{eq:identifier}
\hat{\lambda}(z,\overline{z})=g(z)\overline{g(z)}-1.
\end{equation}

\begin{lem}\label{lem:non-deg}
A  singular point $p\in S(f)$ is non-degenerate if and only if $g_z(p)\neq0$.
\end{lem}
\begin{proof}
By a direct calculation, we have the assertion (cf. \cite{maxface}).
\end{proof}
We suppose that any singular point is non-degenerate in the following. 
Then there exists a singular curve $\gamma(t)$ such that $\hat{\lambda}(\gamma(t))=0$. 
It is known that the vector fields $\xi$ which is tangent to $\gamma$ can be taken as 
\begin{equation}\label{eq:xi}
\xi=ig\overline{g}_{\overline{z}}=ig\overline{g}_{\overline{z}}\partial_z-ig_z\overline{g}\partial_{\overline{z}}\quad 
(\xi_{\gamma}=i\overline{(g_z/g)}\partial_z-i(g_z/g)\partial_{\overline{z}})
\end{equation}
near $p$ (see \cite{fsuy,maxface}). 
Here we used the following identification:
\begin{equation}\label{eq:identify}
\zeta=a+ib\in\C\leftrightarrow (a,b)\in\R^2\leftrightarrow a\partial_u+b\partial_v\leftrightarrow \zeta\partial_z+\overline{\zeta}\partial_{\overline{z}}.
\end{equation}
We sometimes use the following relation:
\begin{equation}\label{eq:g/gz}
\overline{\left(\dfrac{g_z}{g}\right)}=\dfrac{g}{g_z}\left|\dfrac{g_z}{g}\right|^2
\end{equation}
near $p$. 
Moreover, the null vector field $\eta$ is taken as 
\begin{equation}\label{eq:eta}
\eta=\dfrac{i}{g\hat{\omega}}\partial_z-\dfrac{i}{\overline{g}\overline{\hat{\omega}}}\partial_{\overline{z}}
\end{equation}
(see \cite{fsuy,maxface}). 

\begin{lem}[cf. {\cite[Theorem 3.1]{maxface}}]\label{lem:first-kind}
A non-degenerate singular point $p$ of a surface $f\colon D\to\R^3$ constructed by \eqref{eq:W-rep} is of the first kind 
if and only if $\im(g_z/g^2\hat{\omega})\neq0$ at $p$.
\end{lem}
\begin{proof}
By the identification \eqref{eq:identify}, we identify $\xi$ and $\eta$ with 
$\xi=i\overline{(g_z/g)}$ and $\eta=i/g\hat{\omega}$, respectively. 
Then 
\begin{equation}\label{eq:xi-eta}
\det(\xi,\eta)(p)=\im(\overline{\xi}\eta)(p)=\im\left(\dfrac{g_z}{g^2\hat{\omega}}\right)(p)
\end{equation}
since $\overline{g}=1/g$ at $p$. 
Thus we have the conclusion.
\end{proof}

For surfaces given by \eqref{eq:W-rep} with the holomorphic data $(g,\omega=\hat{\omega}dz)$, 
the following criteria for singularities of them are known.

\begin{fact}[\cite{fsuy,ot,maxface}]\label{fact:crit-sing-max}
Let $f$ be a surface given by \eqref{eq:W-rep} with the holomorphic data $(g,\hat{\omega}dz)$. 
Let $p$ be a non-degenerate singular point of $f$. 
Then the following assertions hold.
\begin{enumerate}
\item\label{max-front} $f$ is a front at $p$ if and only if $\re(g_z/g^2\hat{\omega})\neq0$ at $p$. 
\item\label{max-ce} $f$ at $p$ is a cuspidal edge if and only if $\re(g_z/g^2\hat{\omega})\neq0$ and $\im(g_z/g^2\hat{\omega})\neq0$ at $p$. 
\item\label{max-sw} $f$ at $p$ is a swallowtail if and only if $\im(g_z/g^2\hat{\omega})=0$, $\re((g_z/g^2\hat{\omega}))\neq0$ and 
$$\re\left(\dfrac{g}{g_z}\left(\dfrac{g_z}{g^2\hat{\omega}}\right)_z\right)\neq0$$
at $p$.
\item\label{max-cbf} $f$ at $p$ is a cuspidal butterfly if and only if $\im(g_z/g^2\hat{\omega})=0$, $\re((g_z/g^2\hat{\omega}))\neq0$, 
$$\re\left(\dfrac{g}{g_z}\left(\dfrac{g_z}{g^2\hat{\omega}}\right)_z\right)=0\quad \textit{and} \quad
\im\left(\dfrac{g}{g_z}\left(\dfrac{g}{g_z}\left(\dfrac{g_z}{g^2\hat{\omega}}\right)_z\right)_z\right)\neq0
$$
at $p$. 
\item\label{max-ccr} $f$ at $p$ is a cuspidal cross cap if and only if $\im(g_z/g^2\hat{\omega})\neq0$, $\re((g_z/g^2\hat{\omega}))=0$ and 
$$\im\left(\dfrac{g}{g_z}\left(\dfrac{g_z}{g^2\hat{\omega}}\right)_z\right)\neq0$$
at $p$.
\item\label{max-cs1} $f$ at $p$ is a cuspidal $S_1^{-}$ singularity if and only if 
$\im(g_z/g^2\hat{\omega})\neq0$, $\re((g_z/g^2\hat{\omega}))=0$, 
$$\im\left(\dfrac{g}{g_z}\left(\dfrac{g_z}{g^2\hat{\omega}}\right)_z\right)=0\quad \textit{and} \quad
\re\left(\dfrac{g}{g_z}\left(\dfrac{g}{g_z}\left(\dfrac{g_z}{g^2\hat{\omega}}\right)_z\right)_z\right)\neq0
$$
at $p$. 
Moreover, there are no surfaces given by \eqref{eq:W-rep} with cuspidal $S_1^{+}$ singularity. 
\end{enumerate}
\end{fact}
We note that types of singularities of maxfaces in $\R^3_1$ are characterized by this fact. 
Furthermore, we have a stronger result than the last statement of \ref{max-cs1} in Fact \ref{fact:crit-sing-max}.
\begin{thm}\label{prop:no-cSk}
For $k\geq2$, there are no surfaces given by \eqref{eq:W-rep} with cuspidal $S_k^{\pm}$ singularities.
\end{thm}
\begin{proof}
Let $f\colon D\to\R^3$ be a surface given by \eqref{eq:W-rep} with the holomorphic data $(g,\omega=\hat{\omega}dz)$. 
Let $p$ be a singular point of the first kind of $f$ and $\gamma(t)$ $(t\in(-\eps,\eps))$ a singular curve through $p(=\gamma(0))$. 
Then we set a function $\psi\colon(-\eps,\eps)\to\R$ by 
$$\psi(t)=\det(\hat{\gamma}'(t),\n\circ\gamma(t),d\n_{\gamma(t)}(\eta(t))),$$
where $\hat{\gamma}=f\circ\gamma$, $\n$ is the Euclidean Gauss map of $f$ as in \eqref{eq:normal} and $\eta$ is a null vector field as in \eqref{eq:eta}. 
By a direct calculation, we see that 
$$\psi=-|\hat{\omega}|^2\im\left(\dfrac{g_z}{g^2\hat{\omega}}\right)\re\left(\dfrac{g_z}{g^2\hat{\omega}}\right)$$
holds at $p$. 
We assume that $f$ is not a front at $p$, that is, $\re(g_z/g^2\hat{\omega})(p)=0$ (see Fact \ref{fact:crit-sing-max}). 
If $f$ has a cuspidal $S_k$ singularity $(k\geq 2)$, then $\psi=\psi'=\psi''=\cdots=\psi^{(k)}=0$ and $\psi^{(k+1)}\neq0$ at $p$ 
(see \cite[Theorem 3.2]{saji}), 
where we take a parameter $t$ satisfying $d/dt=i(\overline{(g_z/g)}\partial_z-(g_z/g)\partial_{\overline{z}})$ (cf. \cite{fsuy,maxface}). 
In particular, $\psi=\psi'=\psi''=0$ at $p$ is equivalent to 
\begin{equation}\label{eq:condition-psi}
\re\left(\dfrac{g_z}{g^2\hat{\omega}}\right)=\im\left(\dfrac{g}{g_z}\left(\dfrac{g_z}{g^2\hat{\omega}}\right)_z\right)=
\re\left(\dfrac{g}{g_z}\left(\dfrac{g}{g_z}\left(\dfrac{g_z}{g^2\hat{\omega}}\right)_z\right)_z\right)=0
\end{equation}
at $p$ by using the relation \eqref{eq:g/gz} (see \cite[Page 124]{ot}). 

On the other hand, the function $a$ which is defined in condition (b) of \cite[Theorem 3.2]{saji} can be written as 
$$a=32|\hat{\omega}|^2\im\left(\dfrac{g_z}{g^2\hat{\omega}}\right)^5
\re\left(\dfrac{g}{g_z}\left(\dfrac{g}{g_z}\left(\dfrac{g_z}{g^2\hat{\omega}}\right)_z\right)_z\right)$$
(see \cite[Page 125]{ot}). 
Thus if $f$ has a cuspidal $S_{k\geq2}$ singularity at $p$, $a$ vanishes at $p$ by \eqref{eq:condition-psi}. 
This implies that $f$ cannot have a cuspidal $S_{k\geq2}$ singularity by \cite[Theorem 3.2]{saji}. 
\end{proof}
By this theorem, one can see that there are no maxfaces with cuspidal $S_{k\geq2}$ singularities.

\subsubsection{The case of a surface given by \eqref{eq:generalized-cmc}}\label{sec:sing-cmc}
We review some notions of surfaces given by \eqref{eq:generalized-cmc}. 
Let $f\colon D\to\R^3$ be a surface given by \eqref{eq:generalized-cmc} 
with an extended harmonic map $g$, 
where $D$ is a simply-connected domain in the complex plane $\C$ with complex coordinate $z=u+iv$. 
We assume that $|g|$ takes finite values on $D$ to focus on non-degenerate singular points. 

By \eqref{eq:generalized-cmc}, the first order differentials of $f$ by $z$ and $\overline{z}$ are
\begin{equation}\label{eq:diff-cmc}
f_z=\dfrac{1}{H}(-2g,1+g^2,i(1-g^2))\hat{\omega},\quad
f_{\overline{z}}=\dfrac{1}{H}(-2\overline{g},1+\overline{g}^2,-i(1-\overline{g}^2))\overline{\hat{\omega}}.
\end{equation} 
Thus we have 
\begin{equation*}
f_u\times f_v=-2if_z\times f_{\overline{z}}=\dfrac{(|g|^2-1)|\hat{\omega}|^2}{H^2}(1+|g|^2,2\re(g),2\im(g)),
\end{equation*}
and hence 
the unit normal vector $\n$ of $f$ can be taken as 
\begin{equation}\label{eq:normal-cmc}
\n=\dfrac{1}{\sqrt{(1+|g|^2)^2+4|g|^2}}(1+|g|^2,2\re(g),2\im(g)).
\end{equation}

Using $f_z$, $f_{\overline{z}}$ and $\n$, the signed area density function of $f$ is 
\begin{equation}\label{eq:lambda-cmc}
\lambda=(|g|^2-1)|\hat{\omega}|^2\dfrac{\sqrt{(1+|g|^2)^2+4|g|^2}}{H^2}.
\end{equation}
Since $\sqrt{(1+|g|^2)^2+4|g|^2}/H^2>0$, the set of singular points $S(f)$ of $f$ is the union $S(f)=S_1(f)\cup S_2(f)$, 
where 
$$S_1(f)=\{p\in D\ |\ |g(p)|-1=0\},\quad S_2(f)=\{p\in D\ | \ |\hat{\omega}(p)|=0\}.$$
By assumption, $\hat{\omega}\neq0$ (see Definition \ref{def:generalized-cmc}). 
Thus $S_2(f)=\emptyset$ in such a case. 
Moreover, the singularity identifier $\hat{\lambda}$ is $\hat{\lambda}(z)=g(z)\overline{g(z)}-1$. 
\begin{lem}[{\cite{umeda}}]\label{lem:nondeg-cmc}
Under the above setting, a singular point $p$ of $f$ is non-degenerate 
if and only if $g_z(p)\neq0$.
\end{lem}

Let us assume that a point $p\in S_1(f)$ is a non-degenerate singular point of $f$. 
Then by similar discussions for the case of surfaces given by \eqref{eq:W-rep}, 
we can take vector fields $\xi$ and $\eta$ as 
\begin{equation}\label{eq:cmc-xi}
\xi=i\overline{g}_{\overline{z}}{g}\partial_z-i{g_z}\overline{g}\partial_{\overline{z}},\quad
\eta=\dfrac{i}{g\hat{\omega}}\partial_z-\dfrac{i}{\overline{g}\overline{\hat{\omega}}}\partial_{\overline{z}},
\end{equation}
which are the singular direction along the singular curve $\gamma$ and a null vector field, respectively (cf. \cite{umeda}).
Thus we have the following.
\begin{lem}[{\cite[Theorem 4.1]{umeda}}]\label{lem:first-kind-cmc}
A non-degenerate singular point $p$ of a surface $f$ given by \eqref{eq:generalized-cmc} with an extended harmonic map $g$ 
is of the first kind if and only if $\im(g_z/g^2\hat{\omega})\neq0$ at $p$.
\end{lem}

The following characterizations of singularities are known. 
\begin{fact}[{\cite[Theorem 4.1]{umeda}}]\label{fact:crit-cmc}
Let $f$ be a surface given by \eqref{eq:generalized-cmc} with an extended harmonic map $g$. 
Let $p$ be a non-degenerate singular point of $f$. 
Then the following assertions hold.
\begin{enumerate}
\item $f$ at $p$ is a front if and only if $\re(g_z/g^2\hat{\omega})\neq0$ at $p$. 
\item $f$ at $p$ is a cuspidal edge if and only if $\re(g_z/g^2\hat{\omega})\neq0$ and $\im(g_z/g^2\hat{\omega})\neq0$ at $p$.
\item $f$ at $p$ is a swallowtail if and only if $\re(g_z/g^2\hat{\omega})\neq0$, $\im(g_z/g^2\hat{\omega})=0$ and 
$$\re\left(\dfrac{g}{g_z}\left(\dfrac{g_z}{g^2\hat{\omega}}\right)_z\right)\neq
\re\left(\overline{\left(\dfrac{g}{g_z}\right)}\left(\dfrac{g_z}{g^2\hat{\omega}}\right)_{\overline{z}}\right)$$
at $p$. 
\item $f$ at $p$ is a cuspidal cross cap if and only if $\re(g_z/g^2\hat{\omega})=0$, $\im(g_z/g^2\hat{\omega})\neq0$ and 
$$\im\left(\dfrac{g}{g_z}\left(\dfrac{g_z}{g^2\hat{\omega}}\right)_z\right)\neq
\im\left(\overline{\left(\dfrac{g}{g_z}\right)}\left(\dfrac{g_z}{g^2\hat{\omega}}\right)_{\overline{z}}\right)$$
at $p$. 
\end{enumerate}
\end{fact}

This fact characterizes types of singularities on extended CMC surfaces in $\R^3_1$. 
We shall extend this result under some additional assumption. 
\begin{thm}\label{thm:crit-cbf}
Let $f$ be a surface given by \eqref{eq:generalized-cmc} with an extended harmonic map $g$. 
Let $p$ be a non-degenerate singular point of $f$. 
Assume that $\hat{\omega}$ as in \eqref{eq:cmc-omega} can be extended to a function of 
at least class $C^2$ across $S_1(f)=\{p\in D\ |\ |g(p)|=1\}$. 
Then $f$ at $p$ is a cuspidal butterfly if and only if 
$\re(g_z/g^2\hat{\omega})\neq0$, $\im(g_z/g^2\hat{\omega})=0$, 
$$\re\left(\dfrac{g}{g_z}\left(\dfrac{g_z}{g^2\hat{\omega}}\right)_z\right)=
\re\left(\overline{\left(\dfrac{g}{g_z}\right)}\left(\dfrac{g_z}{g^2\hat{\omega}}\right)_{\overline{z}}\right)$$
and 
$$\im\left(\dfrac{g}{g_z}\left(\dfrac{g}{g_z}\left(\dfrac{g_z}{g^2\hat{\omega}}\right)_z\right)_z\right)
+\im\left(\overline{\left(\dfrac{g}{g_z}\right)}\left(\overline{\left(\dfrac{g}{g_z}\right)}\left(\dfrac{g_z}{g^2\hat{\omega}}\right)_{\overline{z}}\right)_{\overline{z}}\right)
\neq\dfrac{1}{|g_z|^2}\im\left(\left(\dfrac{g_z}{g^2\hat{\omega}}\right)_{z\overline{z}}\right)$$
hold at $p$.
\end{thm} 
\begin{proof}
Let $\gamma(t)$ be a singular curve passing through $p$. 
Take a parameter $t$ satisfying the relation  $d/dt=i(\overline{(g_z/g)}\partial_z-(g_z/g)\partial_{\overline{z}})$ (cf. \cite{fsuy,umeda,maxface}). 
Then we set a function $\delta$ as 
$$\delta(t)=\det(\gamma',\eta)(t),$$
where $\eta$ is as in \eqref{eq:cmc-xi}. 
In this case, $\delta$ can be written as 
$$\delta=\im\left(\dfrac{g_z}{g^2\hat{\omega}}\right)=\im(\phi)=\dfrac{1}{2i}(\phi-\overline{\phi}),$$
where we set $\phi=g_z/g^2\hat{\omega}$. 
By \cite[Corollary A. 9]{is}, $f$ at a non-degenerate singular point $p$ is a cuspidal butterfly 
if and only if $f$ at $p$ is a front and $\delta(0)=\delta'(0)=0$ but $\delta''(0)\neq0$. 
Thus we calculate the first and the second order derivatives of $\delta$ by $t$. 
By the above expression and the relation, we have 
$$\delta'=\left(\re\left(\dfrac{g}{g_z}\phi_z\right)-\re\left(\overline{\left(\dfrac{g}{g_z}\right)}\phi_{\overline{z}}\right)\right)\left|\dfrac{g_z}{g}\right|^2
=\tilde{\delta}\left|\dfrac{g_z}{g}\right|^2,$$
where we used the relation as in \eqref{eq:g/gz}. 

We now suppose that $\delta'(0)=0$. 
Then $\delta''(0)\neq0$ is equivalent to $\tilde{\delta}'(0)\neq0$. 
Hence we calculate $\tilde{\delta}'$. 
By a direct computation, we see that 
\begin{align*}
\tilde{\delta}'&=
\dfrac{i}{2}\dfrac{g}{g_z}\left(\left(\dfrac{g}{g_z}\phi_z\right)_z+\overline{\left(\dfrac{g}{g_z}\right)}\overline{\phi}_{z\overline{z}}
-\overline{\left(\dfrac{g}{g_z}\right)}\phi_{z\overline{z}}-\left(\dfrac{g}{g_z}\overline{\phi}_z\right)_z\right)\left|\dfrac{g_z}{g}\right|^2\\
&\quad -\dfrac{i}{2}\overline{\left(\dfrac{g}{g_z}\right)}
\left(\dfrac{g}{g_z}\phi_{z\overline{z}}+\overline{\left(\left(\dfrac{g}{g_z}\phi_z\right)_z\right)}
-\left(\overline{\left(\dfrac{g}{g_z}\right)}\phi_{\overline{z}}\right)_{\overline{z}}-\dfrac{g}{g_z}\overline{\phi}_{z\overline{z}}\right)\left|\dfrac{g_z}{g}\right|^2\\
&=-\im\left(\dfrac{g}{g_z}\left(\dfrac{g}{g_z}\phi_z\right)_z\right)\left|\dfrac{g_z}{g}\right|^2
-\im\left(\overline{\left(\dfrac{g}{g_z}\right)}\left(\overline{\left(\dfrac{g}{g_z}\right)}\phi_{\overline{z}}\right)_{\overline{z}}\right)\left|\dfrac{g_z}{g}\right|^2
+\left|\dfrac{g}{g_z}\right|^2\im(\phi_{z\overline{z}})\left|\dfrac{g_z}{g}\right|^2\\
&=|g_z|^2\left(-\im\left(\dfrac{g}{g_z}\left(\dfrac{g}{g_z}\phi_z\right)_z\right)
-\im\left(\overline{\left(\dfrac{g}{g_z}\right)}\left(\overline{\left(\dfrac{g}{g_z}\right)}\phi_{\overline{z}}\right)_{\overline{z}}\right)
+\dfrac{1}{|g_z|^2}\im(\phi_{z\overline{z}})\right)
\end{align*}
holds at $p$. 
Therefore we have the conclusion. 
\end{proof}

In \cite{brander}, Brander gave the Bj\"{o}rling formula for spacelike CMC surfaces and investigated singularities. 
We remark that criteria for a cuspidal edge, a swallowtail and a cuspidal cross cap are known in terms of the Bj\"{o}rling data (\cite{brander}). 
Moreover, a criterion for a cuspidal butterfly by the Bj\"{o}rling data is known (\cite{naka}). 

\section{Geometric properties of surfaces given by certain representation formulae}\label{sec:property}
In this section, we study geometric properties of surfaces given by 
\eqref{eq:W-rep} or \eqref{eq:generalized-cmc} near singular points of the first kind. 

\subsection{Curvatures along singular curves}\label{sec:proof}
We show the following assertion related to shapes of surfaces given by 
\eqref{eq:W-rep} or \eqref{eq:generalized-cmc} at singular points of the first kind. 
\begin{thm}\label{thm:ks}
Let $D$ be a simply-connected domain in $\C$ and 
let $f\colon D\to\R^3$ be a frontal surface given by \eqref{eq:W-rep} $($resp. \eqref{eq:generalized-cmc}$)$ 
with the holomorphic data $(g,\omega)$ $($resp. an extended harmonic map $g)$ on $D$. 
Then the singular curvature $\kappa_s$ of $f$ is strictly negative at singular points of the first kind, 
and the limiting normal curvature $\kappa_\nu$ vanishes at non-degenerate singular points.
\end{thm}
\begin{proof}
We first give a proof for the case of a surface given by \eqref{eq:W-rep}. 
Let $D$ be a simply-connected domain in $\C$. 
Let $f\colon D\to\R^3$ be a surface given by \eqref{eq:W-rep} with the holomorphic data $(g,\omega=\hat{\omega}dz)$ on $D$. 
Then we consider the first order directional derivative of $f$ in the direction $\xi$. 
By \eqref{eq:df} and \eqref{eq:xi}, we have 
\begin{align}\label{eq:xi-f}
\begin{aligned}
\xi f&=i\overline{\left(\dfrac{g_z}{g}\right)}f_z-i\dfrac{g_z}{g}f_{\overline{z}}
=i\left(\overline{\left(\dfrac{g_z}{g^2}\right)}\hat{\omega}-\dfrac{g_z}{g^2}\overline{\hat{\omega}}\right)(-1,\re(g),\im(g))\\
&=i\left(\overline{\left(\dfrac{g_z}{g^2\hat{\omega}}\right)}-\dfrac{g_z}{g^2\hat{\omega}}\right)|\hat{\omega}|^2|(-1,\re(g),\im(g))
=2\im\left(\dfrac{g_z}{g^2\hat{\omega}}\right)|\hat{\omega}|^2(-1,\re(g),\im(g))
\end{aligned}
\end{align}
along $\gamma$. 
Thus we see that 
\begin{equation}\label{eq:norm}
|\xi f|=2\sqrt{2}\left|\im\left(\dfrac{g_z}{g^2\hat{\omega}}\right)\right||\hat{\omega}|^2.
\end{equation}
Moreover, by \eqref{eq:normal} and \eqref{eq:xi-f}, it holds that 
\begin{equation}\label{eq:cross-xif}
\hat{\n}\times \xi f=\dfrac{4}{\sqrt{2}}\im\left(\dfrac{g_z}{g^2\hat{\omega}}\right)|\hat{\omega}|^2(0,-\im(g),\re(g))
\end{equation}
along $\gamma$, where $\hat{\n}$ is 
\begin{equation}\label{eq:normal2}
\hat{\n}(t)=\n(\gamma(t))=\dfrac{1}{2\sqrt{2}}(2,2\re(g),2\im(g))=\dfrac{1}{\sqrt{2}}(1,\re(g),\im(g)).
\end{equation} 

We next consider the second order directional derivative $\xi\xi f$. 
Since $\xi=ig\overline{g_z}$, we see that 
\begin{align*}
\xi\xi f&=-g\overline{g}_{\overline{z}}(g_z\overline{g}_{\overline{z}}f_z+g\overline{g}_{\overline{z}}f_{zz}-g_{zz}\overline{g}f_{\overline{z}})
+g_z\overline{g}(g\overline{g}_{\overline{zz}}f_z-g_z\overline{g}_{\overline{z}}f_{\overline{z}}-g_z\overline{g}f_{\overline{z}\overline{z}})\\
&=(g_z\overline{g}_{\overline{zz}}-gg_{z}\overline{g}_{\overline{z}}^2)f_z+(\overline{g}_{\overline{z}}g_{zz}-\overline{g}\overline{g}_{\overline{z}}g_z^2)f_{\overline{z}}
-(g^2\overline{g}_{\overline{z}}^2f_{zz}+\overline{g}^2g_z^2f_{\overline{z}\overline{z}})
\end{align*}
along $\gamma$. 
Setting $X=g_z\overline{g}_{\overline{zz}}-gg_{z}\overline{g}_{\overline{z}}^2$, it follows that 
$$Xf_{z}+\overline{X}f_{\overline{z}}=\re(Xg\hat{\omega})(-1,\re(g),\im(g)).$$
This is perpendicular to $\hat{\n}$ and parallel to $\xi f$. 
We calculate  $g^2\overline{g_z}^2f_{zz}+\overline{g}^2g_z^2f_{\overline{z}\overline{z}}$. 
By \eqref{eq:df}, we have 
\begin{align}\label{eq:ddf}
\begin{aligned}
f_{zz}&=(-1,g,-ig)g_z\hat{\omega}+\dfrac{1}{2}(-2g,1+g^2,i(1-g^2))\hat{\omega}_z,\\
f_{\overline{z}\overline{z}}&=(-1,\overline{g},i\overline{g})\overline{g}_{\overline{z}}\overline{\hat{\omega}}
+\dfrac{1}{2}(-2\overline{g},1+\overline{g}^2,-i(1-\overline{g}^2))\overline{\hat{\omega}}_{\overline{z}}
\end{aligned}
\end{align}
at $p$. 
Thus we have 
\begin{align*}
g^2\overline{g}_{\overline{z}}^2f_{zz}+g_z^2\overline{g}^2f_{\overline{z}\overline{z}}
&=\re(g^3\overline{g}_{\overline{z}}\hat{\omega}_z)\psi
+|g_z|^2|\hat{\omega}|^2\re\left(\overline{\left(\dfrac{g_z}{g^2\hat{\omega}}\right)}(-1,g,-ig)\right)
\end{align*}
along $\gamma$. 
Here we set $\phi=(g_z/g^2\hat{\omega})$ and $\psi=(-1,\re(g),\im(g))$. 
Then 
\begin{equation}\label{eq:xixif}
\xi\xi f=Y\psi-|g_z|^2|\hat{\omega}|^2(-\overline{\phi}-\phi,g\overline{\phi}+\overline{g}\phi,-i(g\overline{\phi}-\overline{g}\phi))
\end{equation}
holds, where $Y$ is a some function. 
It is obvious that $\inner{\psi}{\hat{\n}}=\det(\psi,\xi f,\hat{\n})=0$ at a singular point $p$. 
Therefore by \eqref{eq:cross-xif} and \eqref{eq:xixif}, we have 
\begin{equation}\label{eq:determinant}
\det(\xi f,\xi\xi f,\hat{\n})=\inner{\hat{\n}\times\xi f}{\xi\xi f}=\dfrac{8}{\sqrt{2}}\left(\im\left(\dfrac{g_z}{g^2\hat{\omega}}\right)\right)^2|g_z|^2|\hat{\omega}|^4
\end{equation}
along the singular curve $\gamma$. 

On the other hand, by \eqref{eq:identifier} and \eqref{eq:eta}, we have 
\begin{equation}\label{eq:dir-lambda}
\eta\hat{\lambda}=\dfrac{i}{g\hat{\omega}}\hat{\lambda}_z-\dfrac{i}{\overline{g}\overline{\hat{\omega}}}\hat{\lambda}_{\overline{z}}
=i\left(\dfrac{g_z}{g^2\hat{\omega}}-\overline{\left(\dfrac{g_z}{g^2\hat{\omega}}\right)}\right)=-2\im\left(\dfrac{g_z}{g^2\hat{\omega}}\right)
\end{equation}
at a singular point $p$, 
and hence we get 
\begin{equation}\label{eq:sign}
\eps_{\gamma}=\sgn\left(-2\left(\im\left(\dfrac{g_z}{g^2\hat{\omega}}\right)\right)^2\right)=-1
\end{equation}
by \eqref{eq:xi-eta} and \eqref{eq:dir-lambda}. 
Thus we have 
\begin{equation}\label{eq:kappas}
\kappa_s(\gamma)=\left.\eps_{\gamma}\dfrac{\det(\xi f,\xi\xi f,\n)}{|\xi f|^3}\right|_{\gamma}
=\left.-\dfrac{|g_z|^2}{4\left|\im\left(\dfrac{g_z}{g^2\hat{\omega}}\right)\right||\hat{\omega}|^2}\right|_{\gamma}
\end{equation}
along $\gamma$ by \eqref{eq:norm}, \eqref{eq:determinant} and \eqref{eq:sign}. 
This implies that $\kappa_s$ is strictly negative. 

Further we consider the limiting normal curvature $\kappa_\nu$. 
By \eqref{eq:normal2} and \eqref{eq:xixif}, 
\begin{align*}
\inner{\xi\xi f}{\hat{\n}}&=-\dfrac{|\hat{\omega}|^2}{\sqrt{2}}(-\phi-\overline{\phi}+\re(g)(g\overline{\phi}+\overline{g}\phi)-i\im(g)(g\overline{\phi}-\overline{g}\phi))\\
&=-\dfrac{|\hat{\omega}|^2}{\sqrt{2}}\left(
-\phi-\overline{\phi}+\dfrac{1}{2}(g^2\overline{\phi}+\phi+\overline{\phi}+\overline{g}^2\phi)
-\dfrac{1}{2}(g^2\overline{\phi}-\phi-\overline{\phi}+\overline{g}^2\phi)
\right)\\
&=0
\end{align*}
holds at a singular point $p$. 
This implies that $\kappa_\nu=0$ along the singular curve $\gamma$ (see \eqref{eq:invariants}). 
Therefore we have the assertion for the case of a surface given by \eqref{eq:W-rep}.

We next consider the case for a surface given by \eqref{eq:generalized-cmc}. 
Let $f\colon D\to\R^3$ be a surface given by \eqref{eq:generalized-cmc} with an extended harmonic map $g$ on $D$. 
In this case, the Gauss map $\n$ is given by \eqref{eq:normal-cmc}.
Then we calculate the first and the second order directional derivatives in the direction $\xi$ as in \eqref{eq:cmc-xi}. 
By \eqref{eq:diff-cmc} and the relation $\overline{g}=1/g$ on the set of singular points $S_1(f)$, it follows that 
\begin{align}\label{eq:cmc-xif}
\begin{aligned}
\xi f&=i\overline{\left(\dfrac{g_z}{g}\right)}f_z-i\left(\dfrac{g_z}{g}\right)f_{\overline{z}}
=\dfrac{2i}{H}\left(\overline{\left(\dfrac{g_z}{g^2\hat{\omega}}\right)}-\dfrac{g_z}{g^2\hat{\omega}}\right)|\hat{\omega}|^2(-1,\re(g),\im(g))\\
&=\dfrac{4|\hat{\omega}|^2}{H}\im\left(\dfrac{g_z}{g^2\hat{\omega}}\right)(-1,\re(g),\im(g))
\end{aligned}
\end{align}
on $S_1(f)$. 
In particular, if $p\in S_1(f)$ is of the first kind, then $\xi f$ does not vanish at $p$ by Lemma \ref{lem:first-kind-cmc}. 
By \eqref{eq:diff-cmc}, 
\begin{align*}
f_{zz}&
=\dfrac{2g_z\hat{\omega}}{H}(-1,g,-ig)+\dfrac{2g\hat{\omega}_z}{H}(-1,\re(g),\im(g)),\\
f_{z\overline{z}}&=\dfrac{2g\hat{\omega}_{\overline{z}}}{H}(-1,\re(g),\im(g)),\quad
f_{\overline{z}z}=\dfrac{2\overline{g}\overline{\hat{\omega}}_z}{H}(-1,\re(g),\im(g)),\\
f_{\overline{z}\overline{z}}&=\dfrac{2\overline{g}_{\overline{z}}\overline{\hat{\omega}}}{H}(-1,\overline{g},i\overline{g})
+\dfrac{2\overline{g}\overline{\hat{\omega}}_{\overline{z}}}{H}(-1,\re(g),\im(g))
\end{align*}
hold at $p\in S_1(f)$ since $g_{\overline{z}}=\overline{g}_z=0$ at $p$. 
Therefore we have 
\begin{equation}\label{eq:cmc-xixif}
\xi\xi f=Z(-1,\re(g),\im(g))
-\dfrac{2|g_z|^2|\hat{\omega}|^2}{H}(-\phi-\overline{\phi},g\overline{\phi}+\overline{g}\phi,-i(g\overline{\phi}-\overline{g}\phi))
\end{equation}
holds on $S_1(f)$, where $Z$ is a some function and $\phi=g_z/g^2\hat{\omega}$. 

We set $\hat{\n}=\n\circ\gamma$, where $\gamma$ is a singular curve through $p\in S_1(f)$. 
This is expressed as \eqref{eq:normal2}. 
Since $(-1,\re(g),\im(g))$ is perpendicular to $\hat{\n}$, we have $\inner{\xi\xi f}{\hat{\n}}=0$. 
This implies that the limiting normal curvature $\kappa_\nu$ vanishes identically along $\gamma$. 
Moreover, by \eqref{eq:cmc-xif} and \eqref{eq:normal2}, we obtain 
\begin{equation}\label{eq:cross-cmc}
\hat{\n}\times \xi f=\dfrac{8}{\sqrt{2}H}\im\left(\dfrac{g_z}{g^2\hat{\omega}}\right)|\hat{\omega}|^2(0,-\im(g),\re(g))
\end{equation}
along $\gamma$. 
Thus by \eqref{eq:cmc-xixif} and \eqref{eq:cross-cmc}, it follows that 
\begin{equation}\label{eq:cmc-det}
\det(\xi f,\xi\xi f,\hat{\n})=\inner{\hat{\n}\times \xi f}{\xi\xi f}
=\dfrac{16}{\sqrt{2}H^2}\left(\im\left(\dfrac{g_z}{g^2\hat{\omega}}\right)\right)^2|g_z|^2|\hat{\omega}|^4(>0)
\end{equation}

On the other hand, the singular identifier $\hat{\lambda}$ of $f$ is $\hat{\lambda}=g\overline{g}-1$, 
and the null vector field $\eta$ is $\eta=i/g\hat{\omega}$ (cf. \eqref{eq:cmc-xi}), and hence we have 
$$\eta\hat{\lambda}=-2\im\left(\dfrac{g_z}{g^2\hat{\omega}}\right),\quad \det(\xi,\eta)=\im\left(\dfrac{g_z}{g^2\hat{\omega}}\right)$$
at $p\in S_1(f)$. 
Thus it holds that $\eps_{\gamma}=\sgn(\eta\hat{\lambda}\cdot\det(\xi,\eta))=-1$. 
By \eqref{eq:cmc-det}, $\kappa_s$ is given as 
$$\kappa_s=-\dfrac{|H||g_z|^2}{16\left|\im\left(\dfrac{g_z}{g^2\hat{\omega}}\right)\right||\hat{\omega}|^2}$$
along $\gamma$. 
This completes the proof. 
\end{proof}

We remark that if $p$ is a singular point of the second kind and the singular curve $\gamma(t)$ through $p=\gamma(0)$ 
consists of singular points of the first kind for $t\neq 0$, then the singular curvature behaves $\lim_{t\to0}\kappa_s(t)=-\infty$ (\cite{geomfront}).  

\begin{ex}
Let $f$ be a surface given by \eqref{eq:W-rep} with the holomorphic data $(g,\omega)=(z,dz)$ on $\C$. 
If we regard $f$ as a surface in $\R^3_1$, 
this surface is known as the {\it Lorentzian Enneper surface} (see Figure \ref{fig:one}). 
The set of singular points $S(f)$ is $S(f)=\{|z|=1\}$. 
The points $z=\pm1,\pm i$ are swallowtails and the points $z=e^{\frac{\pi i}{4}},e^{\frac{3\pi i}{4}},e^{\frac{5\pi i}{4}},e^{\frac{7\pi i}{4}}$ 
are cuspidal cross caps (cf. \cite{maxface}). 
Thus the curve $\gamma(t)=e^{i t}$ consists of singular points of the first kind for $t\in(0,\pi/2)\cup(\pi/2,\pi)\cup(\pi,3\pi/2)\cup(3\pi/2,2\pi)$. 
Regard $f$ as a surface in $\R^3$. 
Then we have the singular curvature $\kappa_s$ as 
$$\kappa_s=\dfrac{-1}{4|\sin 2t|}<0$$
along $\gamma(t)$ for $t\in(0,\pi/2)\cup(\pi/2,\pi)\cup(\pi,3\pi/2)\cup(3\pi/2,2\pi)$.

\begin{figure}[htbp]
 \begin{center}
  \includegraphics[width=5cm]{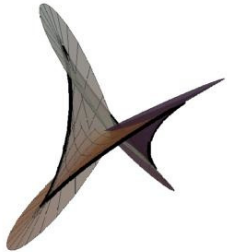}
 \end{center}
 \caption{The Lorentzian Enneper surface. The thick curve is the singular locus.}
 \label{fig:one}
\end{figure}
\end{ex}

\subsection{Behavior of the Gaussian curvature and singularities of the Gauss map}\label{sec:gauss}
We consider behavior of the Gaussian curvature of $f$ given by $\eqref{eq:W-rep}$ or \eqref{eq:generalized-cmc} 
as a surface in $\R^3$. 
First, we investigate the case of a surface given by \eqref{eq:W-rep} with the holomorphic data $(g,\omega)$. 
We consider behavior of the Gauss map $\n$ as in \eqref{eq:normal}. 
The first order derivatives $\n_z$ and $\n_{\overline{z}}$ of $\n$ by $z$ and $\overline{z}$ are 
\begin{align}\label{eq:dn}
\begin{aligned}
\n_z&=\dfrac{g_z}{\sqrt{(1+|g|^2)^2+4|g|^2}}\left(\overline{g}(1+\hat{\rho}(1+|g|^2)),2\hat{\rho}\overline{g}\re(g)+1,2\hat{\rho}\overline{g}\im(g)-i\right),\\
\n_{\overline{z}}&=\dfrac{\overline{g_z}}{\sqrt{(1+|g|^2)^2+4|g|^2}}
\left({g}(1+\hat{\rho}(1+|g|^2)),2\hat{\rho}{g}\re(g)+1,2\hat{\rho}{g}\im(g)+i\right),
\end{aligned}
\end{align}
where 
$$\hat{\rho}=-\dfrac{3+|g|^2}{(1+|g|^2)^2+4|g|^2}.$$
Thus the vector product $\n_z\times \n_{\overline{z}}$ is given as 
\begin{equation}\label{eq:cross-dn}
\n_z\times\n_{\overline{z}}
=\dfrac{2i|g_z|^2(1-|g|^2)}{((1+|g|^2)^2+4|g|^2)^2}(1+|g|^2,2\re(g),2\im(g)),
\end{equation}
and hence we have 
\begin{equation}\label{eq:Lambda}
\Lambda=\det(\n_u,\n_v,\n)=\inner{\n_u\times\n_v}{\n}=-2i\inner{\n_z\times\n_{\overline{z}}}{\n}
=-\dfrac{4|g_z|^2(|g|^2-1)}{((1+|g|^2)^2+4|g|^2)^{3/2}}
\end{equation}
by \eqref{eq:normal} and \eqref{eq:cross-dn}. 
Therefore the Gaussian curvature $K_{\mathrm{E}}$ of $f$ given by \eqref{eq:W-rep} is 
\begin{equation}\label{eq:Gauss-E}
K_{\mathrm{E}}=\dfrac{\Lambda}{\lambda}=-\dfrac{4|g_z|^2}{((1+|g|^2)^2+4|g|^2)^2|\hat{\omega}|^2}
\end{equation}
by \eqref{eq:lambda} and \eqref{eq:Lambda}. 
This implies that $K_{\mathrm{E}}$ is strictly negative near a non-degenerate singular point $p$ of $f$. 

We next consider the case of a surface given by \eqref{eq:generalized-cmc} with an extended harmonic map $g$. 
In this case, by similar calculations, we have 
\begin{equation}\label{eq:dn-cmc}
\n_{z}=\rho^3
\begin{pmatrix}
X(\rho^{-2}-(3+|g|^2)(1+|g|^2))\\ 
\rho^{-2}(g_z+\overline{g}_z)-2X(3+|g|^2)\re(g)\\
\dfrac{\rho^{-2}}{i}(g_z-\overline{g}_z)-2X(3+|g|^2)\im(g)
\end{pmatrix}^T,\quad 
\n_{\overline{z}}=\rho^3
\begin{pmatrix}
\overline{X}(\rho^{-2}-(3+|g|^2)(1+|g|^2))\\
\rho^{-2}(g_{\overline{z}}+\overline{g}_{\overline{z}})-2\overline{X}(3+|g|^2)\re(g)\\
\dfrac{\rho^{-2}}{i}(g_{\overline{z}}-\overline{g}_{\overline{z}})-2\overline{X}(3+|g|^2)\im(g)
\end{pmatrix}^T,
\end{equation}
where we set 
$$\rho=\dfrac{1}{\sqrt{(1+|g|^2)^2+4|g|^2}},\quad X=g_z\overline{g}+g\overline{g}_z,$$
and $\x^T$ is a transposed vector of $\x$.  
Then by \eqref{eq:dn-cmc}, the cross product of $\n_z$ and $\n_{\overline{z}}$ is 
\begin{equation*}\label{eq:cross-dn-cmc}
\n_z\times\n_{\overline{z}}=\dfrac{2(1-|g|^2)(|g_{\overline{z}}|^2-|g_z|^2)}{i((1+|g|^2)^2+4|g|^2)}(1+|g|^2,2\re(g),2\im(g)).
\end{equation*}
Thus the set of singular points of $\n$ is 
$S(\n)=\{p\in D\ |\ |g(p)|=1\}\cup\{p\in D\ |\ |g_z(p)|=|g_{\overline{z}}(p)|\}$. 
Moreover, we have 
\begin{equation}\label{eq:cross-dn-cmc2}
\n_u\times \n_v=-2i\n_z\times\n_{\overline{z}}=\dfrac{-4(1-|g|^2)(|g_{\overline{z}}|^2-|g_z|^2)}{(1+|g|^2)^2+4|g|^2}(1+|g|^2,2\re(g),2\im(g)),
\end{equation}
and hence the Gaussian curvature $K_E$ is 
\begin{equation}\label{eq:EGauss-cmc}
K_E=\dfrac{\det(\n_u,\n_v,\n)}{\det(f_u,f_v,\n)}=\dfrac{4(|g_{\overline{z}}|^2-|g_z|^2)H^2}{((1+|g|^2)^2+4|g|^2)^2|\hat{\omega}|^2}
\end{equation}
by \eqref{eq:lambda-cmc} and \eqref{eq:cross-dn-cmc2}. 
This implies that $K_E$ changes the sign across the set $\{p\in D\ |\ |g_z(p)|=|g_{\overline{z}}(p)|\}$. 
Further, when $p$ is a non-degenerate singular point of $f$, then $g_z(p)\neq0$ and $g_{\overline{z}}(p)=0$. 
Thus $K_E$ is strictly negative at $p$ by \eqref{eq:EGauss-cmc}. 
As a result, we have the following.

\begin{prop}\label{prop:Gaussian}
Let $f\colon D\to\R^3$ be a surface given by \eqref{eq:W-rep} with the holomorphic data $(g,\omega)$ or 
\eqref{eq:generalized-cmc} with an extended harmonic map $g$. 
Let $p$ be a non-degenerate singular point of $f$. 
Then its Gaussian curvature $K_E$ of $f$ is strictly negative at $p$. 
\end{prop}

By Theorem \ref{thm:ks} and Proposition \ref{prop:Gaussian}, 
we have the following assertion immediately. 

\begin{cor}\label{cor:ks-Gauss}
Under the same assumptions as in Proposition \ref{prop:Gaussian}, the sign of the singular curvature $\kappa_s$ 
at singular point of the first kind $p$ of $f$ coincides with 
the sign of the Gaussian curvature $K_E$ of $f$ at $p$.
\end{cor}

\begin{rem}
Let $f\colon D\to\R^3_1$ be a maxface given by the Weierstrass data $(g,\hat{\omega}dz)$. 
Then the (Lorentzian) Gaussian curvature $K_L$ of $f$ is given as
$$K_L=\dfrac{|g_z|^2}{(1-|g|^2)^4|\hat{\omega}|^2}$$
on the set of regular points (cf. \cite{maxface}). 
Thus $K_L$ is non-negative on the set of regular points. 
On the other hand, let $f\colon D\to\R^3_1$ be an extended spacelike CMC $H(\neq0)$ surface with extended harmonic map $g$. 
Then the Gaussian curvature $K_L$ of $f$ is given by 
$$K_L=H^2\left(\left|\dfrac{g_z}{g_{\overline{z}}}\right|^2-1\right)$$
on the set of regular points (see \cite{an,umeda}). 
Thus there are possibilities that $K_L$ takes positive or negative value. 
In particular, $K_L$ is unbounded near a singular points in both cases.
\end{rem}

For a front in $\R^3$ with a cuspidal edge $p$, it follows that if the Gaussian curvature $K_E$ is non-negative near $p$, 
then the singular curvature is non-positive (\cite[Theorem 3.1]{geomfront}). 
However, the inverse of this fact does not hold in general. 
By the above discussions, we can construct several examples of frontal surfaces in $\R^3$ with bounded negative Gaussian curvatures 
and negative singular curvatures along the singular curves by \eqref{eq:W-rep} and \eqref{eq:generalized-cmc} .

We remark that Akamine \cite{akamine} investigated relationships between signs of the singular curvature 
and the (Lorentzian) Gaussian curvature 
for {\it timelike minfaces} which are timelike surfaces with vanishing mean curvature admitting certain singularities. 

We focus on singularities of the Gauss map $\n$. 
By \eqref{eq:Lambda} and \eqref{eq:cross-dn-cmc2}, a singular point $p$ of $f$ is also a singular point of $\n$. 
Thus locally, we may consider $\hat{\lambda}=|g|^2-1$ and $\xi=i\overline{(g_z/g)}$ as a singularity identifier and the singular direction of $\n$, respectively. 
Since $g_z(p)\neq0$, $\n$ has a non-degenerate singular point at $p$. 
\begin{prop}\label{prop:sing-n}
Suppose that the Gauss map $\n$ of a surface $f\colon D\to\R^3$ given by \eqref{eq:W-rep} with the holomorphic data $(g,\omega)$ 
$($resp. by \eqref{eq:generalized-cmc} with an extended harmonic map $g)$ has a non-degenerate singularity at $p$. 
Then $p$ must be a fold of $\n$. 
\end{prop}
Here a {\it fold} is a map germ $h\colon(\R^2,0)\to(\R^2,0)$ which is $\mathcal{A}$-equivalent to the germ 
$(u,v)\mapsto(u,v^2)$ at the origin. 

\begin{proof}
We look for a null vector field $\eta^{\n}$ of $\n$. 
By \eqref{eq:dn} and \eqref{eq:dn-cmc}, we see that 
$$
\n_z=\dfrac{i}{2\sqrt{2}}\left(\dfrac{g_z}{g}\right)(0,\im(g),-\re(g)),\quad
\n_{\overline{z}}=-\dfrac{i}{2\sqrt{2}}\overline{\left(\dfrac{g_z}{g}\right)}(0,\im(g),-\re(g))
$$
hold at $p$. 
Here we used the relation $g_{\overline{z}}=0$ at $p$ for the case of a surface given by \eqref{eq:generalized-cmc}. 
Thus we can take $\eta^{\n}$ as 
\begin{equation}\label{eq:eta-n}
\eta^{\n}=\overline{\left(\dfrac{g_z}{g}\right)}
\leftrightarrow 
\eta^{\n}=\overline{\left(\dfrac{g_z}{g}\right)}\partial_z+{\left(\dfrac{g_z}{g}\right)}\partial_{\overline{z}}
\end{equation}
along the singular curve $\gamma$ through $p$. 

Using the singular direction $\xi$ as in \eqref{eq:xi} and the null vector $\eta^{\n}$ as in \eqref{eq:eta-n}, we have 
$$\det(\xi,\eta^{\n})=\im(\overline{\xi}\eta^{\n})=-\left|\dfrac{g_z}{g}\right|^2=-|g_z|^2\neq0$$
along $\gamma$. 
This implies that $\n$ has a fold at $p$ (see \cite[Proposition 2.1]{whitney}). 
\end{proof}
By this proposition, the curve $\hat{\n}=\n\circ\gamma$ is a regular spherical curve.

\section{A characterization of a fold singular point}\label{sec:fold}
For a $C^\infty$ map $f\colon\R^2\to\R^3$, 
we say that a singular point $p$ is a {\it fold singular point} of $f$ if $f$ is $\mathcal{A}$-equivalent to the germ $(u,v)\mapsto(u,v^2,0)$ at the origin (cf. \cite{hks}). 
We note that a fold singular point is a singular point of the first kind of a frontal $f\colon \R^2\to\R^3$. 
In \cite{fkkrsuyy}, the notion of a {\it fold singularity} for a singular point of a $C^\infty$ map from $\R^2$ to $\R^3$ is introduced, 
and it is shown that a fold singular point is actually a fold singularity for the case of maxfaces (\cite[Lemma 2.17]{fkkrsuyy}). 
In this section, we give a certain characterization of a fold singular point for more general setting. 
Particularly we prove the following.

\begin{thm}\label{thm:crit-fold}
Let $f\colon U(\subset\R^2)\to\R^3$ be a frontal and $p$ a singular point of the first kind of $f$. 
Then $p$ is a fold singular point of $f$ if and only if 
there exists a local coordinate system $(u,v)$ around $p$ such that $f(u,v)=f(u,-v)$.
\end{thm}

Before giving a proof, we prepare some facts which we need (cf. \cite{fsuy}).

\begin{fact}[the division lemma]\label{fact:division}
Let $h\colon U(\subset\R^2)\to\R$ be a $C^\infty$ function. 
If $h$ satisfies $h(u,0)=0$, then there exists a function $b$ on $U$ 
such that $h(u,v)=vb(u,v)$.
\end{fact}

As a corollary of Fact \ref{fact:division}, we see the following.
\begin{fact}\label{fact:division2}
Let $h\colon U\to\R$ be a $C^\infty$ function. 
Then there exist functions $a$ and $b$ such that $h(u,v)=a(u)+vb(u,v)$.
\end{fact}

On the other hand, the following is known.
\begin{fact}[the Whitney lemma]\label{fact:Whitney}
Let $h\colon U\to\R$ be a $C^\infty$ function. 
If $h$ satisfies $h(u,v)=h(u,-v)$, then there exists a function $b$ on $U$ 
such that $h(u,v)=b(u,v^2)$.
\end{fact}

\begin{proof}[Proof of Theorem \ref{thm:crit-fold}]
If $p$ is a fold singular point of $f$, 
then we may assume that $f$ is given by $f(u,v)=(u,v^2,0)$. 
This obviously satisfies the condition that $f(u,v)=f(u,-v)$. 
Thus we show the opposite side in the following.

Assume that there exists a local coordinate system $(u,v)$ around $p$ such that $f(u,v)=f(u,-v)$. 
Without loss of generality, we may assume that $p=(0,0)$. 
First, we note that $S(f)=\{v=0\}$ holds on $U$ 
since $f_v(u,v)=-f_v(u,-v)$, and hence $f_v(u,0)=0$.
Since $\rank df_p=1$, we may assume that $(f_1)_u(p)\neq0$, where $f=(f_1,f_2,f_3)$. 
Then the map $\phi\colon(u,v)\mapsto(s,t)=(f_1(u,v),v)$ gives a coordinate change on the source around $p$. 
Thus one may have 
\begin{equation}\label{eq:function1}
g(s,t)(=f\circ\phi^{-1}(s,t))=(s,g_2(s,t),g_3(s,t)),
\end{equation}
where $g_i$ $(i=2,3)$ are some $C^\infty$ functions of $s$, $t$. 
We note that the map $g$ satisfies $g(s,t)=g(s,-t)$ by the construction. 

On the other hand, by Fact \ref{fact:division2}, there exist functions $a_i(s)$ and $b_i(s,t)$ $(i=2,3)$ 
such that $g_i(s,t)=a_i(s)+tb_i(s,t)$. 
Moreover, since $(g_i)_t(s,0)=0$ $(i=2,3)$, 
there exist functions $\tilde{b}_i(s,t)$ such that $b_i(s,t)=t\tilde{b}_i(s,t)$ by Fact \ref{fact:division}. 
Thus the map $g$ as in \eqref{eq:function1} can be written as 
\begin{equation}\label{eq:function2}
g(s,t)=(s,a_2(s)+t^2\tilde{b}_2(s,t),a_3(s)+t^2\tilde{b}_3(s,t)).
\end{equation}
Further, noticing that $g(s,t)=g(s,-t)$, there exist functions $\hat{b}_i$ $(i=2,3)$ such that 
$\tilde{b}_i(s,t)=\hat{b}_i(s,t^2)$ by Fact \ref{fact:Whitney}. 
Thus the map $g$ as in \eqref{eq:function2} can be written as 
\begin{equation}\label{eq:function3}
g(s,t)=(s,a_2(s)+t^2\hat{b}_2(s,t^2),a_3(s)+t^2\hat{b}_3(s,t^2)).
\end{equation}
We set $\Phi\colon\R^3\to\R^3$ as 
\begin{equation}\label{eq:diffeo1}
\Phi(X,Y,Z)=(X,Y-a_2(X),Z-a_3(X)).
\end{equation}
This gives a local diffeomorphism on $\R^3$.
Composing $g$ as in \eqref{eq:function3} and $\Phi$ as in \eqref{eq:diffeo1}, 
we have 
\begin{equation*}
h(s,t)=\Phi\circ g(s,t)=(s,t^2\hat{b}_2(s,t^2),t^2\hat{b}_3(s,t^2)).
\end{equation*}
Here we remark that either $\hat{b}_2(p)$ or $\hat{b}_3(p)$ does not vanish by non-degeneracy. 
Thus we may suppose that $\hat{b}_2(p)\neq0$. 
In this case, a map $\tau\colon (s,t)\mapsto(x,y)=\left(s,t\sqrt{\left|\hat{b}_2(s,t^2)\right|}\right)$ 
gives a local coordinate change on the source. 
Thus by a coordinate change, we have 
\begin{equation}\label{eq:function4}
k(x,y)(=h\circ\tau^{-1}(x,y))=(x,y^2,y^2B(x,y^2)),
\end{equation}
where $B$ is a some $C^\infty$ function of $x$, $y$. 
Setting a map $\Psi\colon\R^3\to\R^3$ as 
\begin{equation}\label{eq:diffeo2}
\Psi(X,Y,Z)=(X,Y,Z-YB(X,Y)),
\end{equation}
this map $\Psi$ gives a local diffeomorphism. 
By \eqref{eq:function4} and \eqref{eq:diffeo2}, it holds that 
$$\Psi\circ k(x,y)=(x,y^2,0).$$
Therefore we have the conclusion.
\end{proof}

We note that this singular point relates to the real analytic extension of 
{\it zero mean curvature surfaces} in $\R^3_1$ (cf. \cite{fkkruy2,fkkruy,fkkrsuyy}). 
Further, we remark that there are no extended CMC surfaces in $\R^3_1$ admitting fold singular points (\cite[Theorem 1.1]{hks}). 

\begin{acknowledgements}
The authors would like to express their sincere gratitude to Professor Miyuki Koiso 
for fruitful discussions. 
They are also grateful to Professors Atsufumi Honda, Masaaki Umehara and Kotaro Yamada for their valuable advices and comments. 
\end{acknowledgements}



\begin{thebibliography}{99}
\bibitem{akamine}
S. Akamine, {Behavior of the Gaussian curvature of timelike minimal surfaces with singularities}, 
{\em Hokkaido Math. J.} {\bf 48} (2019), 513--535.

\bibitem{an} 
K. Akutagawa and S. Nishikawa, 
{The Gauss map and spacelike surfaces with prescribed mean curvature in Minkowski $3$-space}, 
{\em Tohoku Math. J. (2)} {\bf 42} (1990), 67--82. 

\bibitem{a} 
\newblock { V.~I. Arnol'd,} 
{\em Singularities of Caustics and Wave Fronts}, 
{Math. and its Appl.} {\bf 62}. Kluwer Academic Publishers Group, Dordrecht, 1990.
%
\bibitem{agv} 
\newblock { V.~I. Arnol'd, S.~M. Gusein-Zade and A.~N. Varchenko,} 
\newblock \emph{Singularities of Differentiable Maps}, Vol.1, 
\newblock {Monographs in Mathematics {\bf82}, Birkh{\"a}user, Boston (1985).}
%
\bibitem{brander}
D. Brander, {Singularities of spacelike constant mean curvature surfaces in Lorentz-Minkowski space}, 
{\em Math. Proc. Cambridge Philos. Soc.} {\bf  150} (2011), 527--556.
%
\bibitem{co2} J. Cho and Y. Ogata, {Maximal surface with planar curvature lines}, 
{\em Beitr. Algebra Geom.} {\bf 59} (2018), 465--489.
%
\bibitem{gene-max}
F. J. M. Estudillo and A. Romero, {Generalized maximal surfaces in Lorentz-Minkowski space $\mathbb{L}^3$}, 
{\em Math. Proc. Cambridge Phil. Soc.} {\bf 111} (1992), 515--524. 
%
\bibitem{ferlop} I. Fern\'{a}ndez and F. J. L\'{o}pez, 
{Periodic maximal surfaces in the Lorentz-Minkowski space $\mathbb{L}^3$}, 
{\em Math. Z.} {\bf 256} (2007), 573--601.
%
\bibitem{fkkruy2} 
S. Fujimori, Y. Kawakami, M. Kokubu, W. Rossman, M. Umehara and K. Yamada, 
{Entire zero mean curvature graphs of mixed type in Lorentz-Minkowski 3-space}, 
{\em Q. J. Math.} {\bf 67} (2016), 801--837.
%
\bibitem{fkkruy} S. Fujimori, Y. Kawakami, M. Kokubu, W. Rossman, M. Umehara and K. Yamada, 
{Analytic extension of Jorge-Meeks type maximal surfaces in Lorentz-Minkowski 3-space}, 
{\em Osaka J. Math.} {\bf 54} (2017), 249--272.
%
\bibitem{fkkrsuyy} S. Fujimori, Y. W. Kim, S.-E. Koh, W. Rossman, H. Shin, M. Umehara, K. Yamada and S.-D. Yang, 
{Zero mean curvature surfaces in Lorentz-Minkowski 3-space and 2-dimensional fluid mechanics}, 
{\em Math. J. Okayama Univ.} {\bf 57} (2015), 173--200.
%
\bibitem{fruyy} 
S. Fujimori, W. Rossman, M. Umehara, K. Yamada and S.-D. Yang, 
{New maximal surfaces in Minkowski $3$-space with arbitrary genus and their cousins in de Sitter $3$-space}, 
{\em Results Math.} {\bf 56} (2009), 41--82.
%
\bibitem{fsuy} 
\newblock {S. Fujimori, K. Saji, M. Umehara and K. Yamada,} 
\newblock {Singularities of maximal surfaces}, 
\newblock {\em Math. Z.} {\bf 259} (2008),  827--848.
%
%
\bibitem{hhnsuy} 
\newblock {M. Hasegawa, A. Honda, K. Naokawa, K. Saji, M. Umehara and K. Yamada,}
\newblock {Intrinsic properties of singularities of surfaces}, 
\newblock {\em Internat. J. Math.} {\bf 26} (2015), 34pp. 
%
%
\bibitem{honda}
A. Honda, {Duality of singularities for spacelike CMC surfaces}, {\em Kobe J. Math.} {\bf 34} (2017),  1--11.
%
\bibitem{hks}
\newblock {A. Honda, M. Koiso and K. Saji,}
\newblock {Fold singularities on spacelike CMC surfaces in Lorentz-Minkowski space,}
\newblock {\em Hokkaido Math. J.} {\bf 47} (2018), 245--267.
%
\bibitem{hs}
A. Honda and K. Saji, 
{Geometric invariants of $5/2$-cuspidal edges}, 
{\em Kodai Math. J.} {\bf 42} (2019), 496--525. 

%
\bibitem{ifrt} 
\newblock {S. Izumiya, M.~C. Romero Fuster, M.~A.~S Ruas and F. Tari,} 
\newblock \emph{Differential Geometry from a Singularity Theory Viewpoint}, 
\newblock {World Scientific Publishing Co. (2015).}
%
\bibitem{is} 
\newblock {S. Izumiya and K. Saji,} 
\newblock {The mandala of 
Legendrian dualities for pseudo-spheres in Lorentz-Minkowski space and ``flat'' spacelike surfaces}, 
\newblock {\em J. Singul.} {\bf2} (2010), 92--127. 
%
\bibitem{kenmotsu} K. Kenmotsu, 
{Weierstrass formula for surfaces of prescribed mean curvature}, 
{\em Math. Ann.} {\bf 245} (1979), 85--99.
%
\bibitem{kimyang1} Y. W. Kim and S.-D. Yang, 
{A family of maximal surfaces in Lorentz-Minkowski three-space}, 
{\em Proc. Amer. Math. Soc.} {\bf 134} (2006), 3379--3390.
%
\bibitem{kimyang2} Y. W. Kim and S.-D. Yang, 
{Prescribing singularities of maximal surfaces via a singular Bj{\"o}rling representation formula}, 
{\em J. Geom. Phys.} {\bf 57} (2007), 2167--2177.
%
\bibitem{koba}
\newblock{O. Kobayashi,} 
{Maximal surfaces in the $3$-dimensional Minkowski space $L^3$}, {\em Tokyo J. Math.} {\bf 6} (1983), 297--309.
%
\bibitem{krsuy} 
\newblock {M. Kokubu, W. Rossman, K. Saji, M. Umehara and K. Yamada,} 
\newblock {Singularities of flat fronts in hyperbolic space}, 
\newblock {\em Pacific J. Math.} {\bf 221} (2005), 303--351.
%
\bibitem{ku} M. Kokubu and M. Umehara,
{Orientability of linear Weingarten surfaces, spacelike CMC-$1$ surfaces and maximal surfaces}, 
{\em Math. Nachr.} {\bf 284} (2011), 1903--1918.
%
\bibitem{ms} L. F. Martins and K. Saji, {Geometric invariants of cuspidal edges}, 
{\em Canad. J. Math.} {\bf 68} (2016), 445--462. 
%
\bibitem{msuy} { L. F. Martins, K. Saji, M. Umehara and K. Yamada}, 
{\em Behavior of Gaussian curvature and mean curvature near non-degenerate singular points on wave fronts}, 
Geometry and Topology of Manifolds, 247--281, Springer Proc. Math. Stat., {\bf 154}, Springer, Tokyo, 2016.
%
\bibitem{naka} 
T. Nakashima, {\em Geometric properties of singular points on extended spacelike CMC surfaces in $\mathbb{L}^3$}, 
Master Thesis, Kyushu University, 2020. 
%
\bibitem{nuy} 
\newblock {K. Naokawa, M. Umehara and K. Yamada,} 
\newblock {Isometric deformations of cuspidal edges}, 
\newblock {\em Tohoku. Math. J. (2)} {\bf 68} (2016), 73--90.
%
\bibitem{ot}
Y. Ogata and K. Teramoto, {Duality between cuspidal butterflies and cuspidal $S_1^-$ singularities on maximal surfaces}, 
{\em Note Mat.} {\bf 38} (2018), 115--130.
%
\bibitem{os} 
R. Oset Sinha and K. Saji, {On the geometry of folded cuspidal edges}, 
{\em Rev. Mat. Complut.} {\bf 31} (2018), 627--650. 
%
\bibitem{saji}
K. Saji, {Criteria for cuspidal $S_k$ singularities and their applications}, 
{\em J. G\"{o}kova Geom. Topol. GGT} {\bf 4} (2010), 67--81. 
%
\bibitem{Aksingul} 
\newblock {K. Saji, M. Umehara and K. Yamada,} 
\newblock {$A_k$ singularities of wave fronts}, 
\newblock {\em Math. Proc. Cambridge Philos. Soc.} {\bf146} (2009), 731--746.
%
\bibitem{geomfront} 
\newblock {K. Saji, M. Umehara and K. Yamada,} 
\newblock {The geometry of fronts}, 
\newblock {\em Ann. of Math. (2)} {\bf169} (2009), 491--529.
%
\bibitem {umeda}
Y. Umeda, {Constant-Mean-Curvature surfaces with singularities in Minkowski 3-space}, 
{\em Exp. Math.} {\bf 18} (2009), 311--323. 
%
\bibitem{maxface}
\newblock{M. Umehara and K. Yamada,}
\newblock {Maximal surfaces with singularities in Minkowski space,} 
\newblock {{\em Hokkaido Math. J.} {\bf 35} (2006), 13--40.}
%
\bibitem{whitney} 
H. Whitney, {On singularities of mappings of Euclidean spaces I. Mappings of the plane into the plane}, 
{\em Ann. of Math. (2)} {\bf 62} (1955), 374--410. 
\end{thebibliography}
\end{document}